\documentclass{aptpub}

\authornames{Daniel Rudolf and Mario Ullrich} 
\shorttitle{Comparison of Markov chains} 

\newtheorem{alg}[theorem]{Algorithm}
\newenvironment{algorithm}{\begin{alg}\rm}{\end{alg}}

\newcommand{\ind}{\mathbf{1}}
\newcommand{\vol}{{\rm vol}}
\newcommand{\R}{\mathbb{R}}
\newcommand{\abs}[1]{\left\vert #1 \right\vert}	
\newcommand{\norm}[1]{\left\Vert #1 \right\Vert}	
\newcommand{\scalar}[2]{\langle #1 , #2 \rangle}
\newcommand{\set}[1]{\lbrace #1 \rbrace}
\renewcommand{\P}{\mathbb{P}}
\renewcommand{\l}{\langle}
\renewcommand{\r}{\rangle}
\renewcommand{\a}{\alpha}

\renewcommand{\t}{\theta}
\newcommand{\dint}{\,\text{\rm d}}

\newcommand{\ball}{\mathbb{B}_d} 
\newcommand{\sphere}{\mathbb{S}_{d-1}}

\newcommand{\A}{\mathcal{A}}

\newcommand{\wt}{\widetilde}

\begin{document}

\title{Comparison of hit-and-run, slice sampler\\ and random walk Metropolis} 

\authorone[
Universit\"at G\"ottingen 
]
{Daniel Rudolf}
\addressone{Goldschmidtstra\ss e 7, 37077 G\"ottingen, Germany, Email: daniel.rudolf@uni-goettingen.de}

\authortwo[Universit\"at Linz]
{Mario Ullrich}

\addresstwo{Altenberger Stra\ss e 69, 4040 Linz, Austria, Email: mario.ullrich@jku.at}

\begin{abstract}
Different Markov chains can be used for approximate sampling 
of a distribution given by an unnormalized density function with respect to the Lebesgue measure.
The hit-and-run, (hybrid) slice sampler and random walk Metropolis algorithm are popular tools to simulate such Markov chains.
We develop a general approach to compare the efficiency of these sampling procedures by the 
use of a partial ordering of their Markov operators, the
covariance ordering. In particular, we show that the hit-and-run and the simple slice sampler are more efficient than a hybrid slice sampler 
based on hit-and-run which, itself, 
is more efficient than a (lazy) random walk Metropolis algorithm.
\end{abstract}



\section{Introduction}
In many scenarios of Bayesian statistics, statistical physics 
and other branches of applied
sciences, see \cite{RoRo02,LoVe06-1,Handbook}, it is of interest to sample 
on $\R^d$ with respect to a distribution $\pi$.
In particular, we assume that $\pi$ is given by an unnormalized density.
More precisely, let
$K\subseteq \R^d$ be an open, measurable set
and $\rho \colon K \to (0,\infty)$ be 
a positive, bounded and with respect to the Lebesgue measure an almost 
everywhere continuous 
function with $\int_K \rho(x) \dint x \in (0,\infty)$.
Then,
define the probability measure $\pi$ through $\rho$ by
\[
\pi(A)=\frac{\int_A \rho(x)\dint x}{ \int_K \rho(x)\dint x}
\]
for all
measurable $A\subseteq K$. 

Maybe the most successful approach to approximate $\pi$ is the construction of a suitable Markov chain.
The hit-and-run algorithm, the random walk Metropolis, 
the simple slice sampler 
and hybrid slice sampler provide such construction methods. 
A crucial question is: Which one of these algorithm should be used?

This question is of course related to the speed of convergence of 
the Markov chain sampling and any answer depends very much on the imposed assumptions. 
In general  
it is difficult to derive explicit estimates of this speed
of convergence. But, it might be possible to prove that
one algorithm is better than another. 
This motivates the idea of the comparison of Markov chains.

The first comparison result of Markov chains is due to Peskun \cite{Pe73}. 
There, a partial ordering
on finite state spaces is invented, where one transition kernel has higher order than another one
if the former dominates the latter off the diagonal. This order was later extended by Tierney \cite{Ti98}
to general state spaces. However, for the Markov chains we have in mind it seems not 
possible to use this off-diagonal ordering. 
We consider a partial ordering on the set of linear operators, see \cite[p.~470]{Kr89}. 
In the context of Markov chains this ordering is called covariance ordering, see \cite{Mi01,MiLe09}.
Let $L_2(\pi)$ be the Hilbert space of functions with finite stationary variance and assume that
$P_1,P_2 \colon L_2(\pi) \to L_2(\pi)$ are two self-adjoint linear operators. 
Then we say, $P_1 \leq P_2$ if and only if 
\[
 \scalar{P_1 f}{f}_\pi \geq \scalar{P_2 f }{f}_\pi, \quad f\in L_2(\pi).
\]
Here the inner-product of $L_2(\pi)$ is given by 
\[
 \scalar{f_1}{f_2}_\pi = \int_K f_1(x) f_2(x) \dint \pi(x), \quad f_1,f_2\in L_2(\pi).
\]
The operator $P_1$ is called positive if $P_1 \geq 0$, that is, $\scalar{P_1 f}{f}_\pi \geq 0$ for any $f\in L_2(\pi)$.
Reversible transition kernels of given 
Markov chains induce self-adjoint Markov operators and
we can compare these operators.

Let $P_1,P_2$ be two of such Markov operators and assume that $0\leq P_1 \leq P_2$.
There are a number of consequences for the corresponding Markov chains: 
For example, the 
spectral gap of $P_2$ is smaller than that of $P_1$. 
The spectral gap of a Markov chain is a quantity which is closely related to 
the speed of convergence to $\pi$, see \cite{Ba05,Ru12}.
Another example is concerning  
the stationary asymptotic variance of
sample averages. For $i=1,2$ let
$S_n^{(i)}(f)=\frac{1}{n} \sum_{k=1}^n f(X^{(i)}_k)$ 
with Markov chain $(X_k^{(i)})_{k\in \N}$ 
starting in stationarity corresponding 
to $P_i$ and an arbitrary function $f\in L_2(\pi)$. 
These sample averages give approximations of the mean $Ef=\int_K f \dint \pi$
and, in \cite{Ti98,Mi01} it is observed that, if $P_1 \leq P_2$, then
\[
 \mathbb{V}(f,P_2) \leq \mathbb{V}(f,P_1).
\]
Here $\mathbb{V}(f,P_i):= \lim_{n\to \infty} n \cdot \mathbb{E}\vert S_n^{(i)}(f)-E f \vert^2 $
is the stationary asymptotic 
variance of $S_n^{(i)}(f)$.
In other words, the Markov chain 
of $P_2$ is (asymptotically) more efficient than that of $P_1$. 
For more details to implications of $P_1\leq P_2$ 
we refer to Section~\ref{sec:impl}.

The goal of this article is to compare the hit-and-run algorithm, 
the (lazy) random walk Metropolis, the simple slice sampler
and a hybrid slice sampler based on hit-and-run according to this partial ordering.
To do so we develop a systematic approach for the comparison of
Markov chains which can be written by a suitable two step procedure.

The intuition behind our approach is the following. 
It is a simple and well-known observation that, if the self-adjoint and 
positive Markov operators $P_1$ and $P_2$ satisfy $P_1 P_2 = P_2$, then 
$P_1\le P_2$. 
However, due to the quite restrictive assumptions, this technique can only 
be used in very few cases. 
A useful generalization of this technique, which was invented in a special case 
in~\cite{Ul14}, can be used whenever we have the representation 
$P_i=R \wt P_i R^*$ for certain operators $R$ and $\wt P_i$, $i=1,2$. 
In this case, and if $\wt P_1 \wt P_2 = \wt P_2$, one can also conclude that 
$P_1\le P_2$, see Lemma~\ref{lemma:main}. 
Our comparison inequalities therefore follow once we have established such 
representations for the Markov chains under consideration. 

\subsection*{The algorithms}
Let us briefly explain the algorithms.
Roughly, a transition of the hit-and-run algorithm works as follows.
Choose randomly a line through the current state and sample
according to $\pi$ restricted to this line.
Thus, instead of sampling with respect to $\pi$ on $K\subseteq\R^d$ 
hit-and-run only uses sampling with respect to $\pi$ on
$1$-dimensional lines through the state space, which is feasible in many cases
\footnote{For example, if $\rho$ is a log-concave density, then the restriction 
of $\pi$ to a line has also a log-concave density and one can use different
acceptance/rejection methods to sample such a distribution on the line efficiently. 
For details we refer to \cite[Sect.~2.4.2]{CaRo04} 
see also \cite{LoVe07,Ru13}.}.

In contrast, the simple slice sampler chooses a suitable $d$-dimensional set, a
super-level set of $\rho$, depending on the current state and then, samples 
the next state of the Markov chain uniformly distributed on this super-level set. 
Sampling of the uniform distribution on a $d$-dimensional set is often not efficiently
implementable. This is why this Markov chain is mostly of theoretical interest. 

The hybrid slice sampler we are interested in overcomes this problem by replacing the 
uniform sampling by one step of a hit-and-run algorithm on the super-level set: 
First, choose a line through the current state uniformly at random and then generate the 
next state uniformly distributed on the intersection of that line with the super-level set.
We call this procedure hybrid slice sampler based on hit-and-run.
Intuitively, it is clear 
that the simple slice sampler is better than that hybrid one.
The intuition for the comparison of the hit-and-run algorithm and 
the hybrid slice sampler based on hit-and-run is not so obvious. 
Observe that this particular hybrid sampler can also be interpreted as
choosing first a line and, then, performing a simple slice sampling step according the
distribution of $\pi$ restricted to that line.
This observation leads us to the fact that the hit-and-run algorithm
is better.

Finally, let us consider the random walk Metropolis.
Assume that we have a proposal density $q$ on $\R^d$ 
and let $x\in K$ be the current state. A transition works as follows:
Generate $z\in \R^d$ according the distribution determined by $q$ and
accept $x+z$ as the next state with probability $\min\{1,\rho(x+z)/\rho(x)\}$
if $x+z\in K$. Otherwise stay at $x$.
It is well known, see \cite{Hi98,RuUl13}, that
the random walk Metropolis can be interpreted as a certain (hybrid) slice sampling procedure, which
runs a random walk according to $q$ on the super-level set with uniform limit distribution.
We want to compare the random walk Metropolis and the hybrid slice sampler based on
a hit-and-run.
Thus, the question is whether the uniform hit-and-run step is better than the 
random walk step on the super-level set. It turns out that this is indeed the case.

\subsection*{Main results}
Now let us formulate the main results.
To compare the above Markov chains we develop 
in Section~\ref{sec:aux} a general approach which might be of interest on its own.
There, in Lemma~\ref{lemma:main} conditions for two suitably defined Markov operators $P_1,P_2$
are stated which imply that $P_1 \leq P_2$. This lemma
is the main ingredient for the comparison argument. Its application 
leads us to Theorem~\ref{thm:main}. 
For the Markov operators $M, U, H, S$ of the (lazy) random walk Metropolis with rotational invariant
proposal $q$, the hybrid slice sampler based on hit-and-run, 
the hit-and-run algorithm and the simple slice sampler, respectively, 
we have
\begin{align*}
  M & \leq U \leq H,\\
  M & \leq U \leq S.
 \end{align*}
Thus, the random walk Metropolis is less efficient than the hit-and-run algorithm and simple slice sampler.
The hybrid slice sampler based on hit-and-run we propose lies concerning efficiency in between.

\subsection*{Outline}
The paper is organized as follows. In the next section
we introduce the notation we use, comment on the partial ordering, present consequences
for the Markov chains and state the algorithms we study in detail. 
In Section~\ref{sec:aux} we invent a new approach how to compare Markov chains
with a specific structure. Section~\ref{sec:main} contains the application
of the former developed comparison arguments. Finally, in Section~\ref{sec:con_rem},
we give some concluding remarks and discuss open problems.

 \section{Preliminaries}	
Let $L_2(\pi)$ be the Hilbert space of functions 
$f\colon K \to \R$ with finite norm $\|f\|_\pi=\l f,f\r_\pi^{1/2}$, 
where the inner-product of $f_1, f_2 \in L_2(\pi)$ is denoted by
\[
 \scalar{f_1}{f_2}_\pi = \int_K f_1(x) f_2(x) \, \dint \pi(x).
\]
Let $P$ be a transition kernel on $K$ which is reversible with respect to $\pi$ and
let $(X_n)_{n\in\N}$ be a Markov chain with transition kernel $P$, i.e.
\[
P(x,A) = \P(X_{k+1}\in A \mid X_{k}=x)
\]
almost surely for all $k\in \N$ and $A\subseteq K$.
The corresponding Markov operator, also denoted by $P$, is given by
\begin{equation} \label{eq:MO}
Pf(x) 
= \int_K f(y) P(x,\!\dint y)
\end{equation}
for $f\in L_2(\pi)$. 
Obviously, $P(x,A)=P\ind_A(x)$ for all 
$A\subseteq K$,
where 
$\ind_A$ denotes the indicator function of $A$. 
Note that, by the reversibility, 
$P: L_2(\pi)\to L_2(\pi)$ is self-adjoint.
We say that a (Markov) operator $P$ on $L_2(\pi)$ is \emph{positive} 
if $\scalar{Pf}{f}_\pi\ge0$ for all $f\in L_2(\pi)$.

\subsection{On the ordering and consequences}
With this notation we define on the set of Markov operators the
following partial ordering.
For Markov operators $P_1$ and $P_2$ we write 
\[
P_1 \le P_2
\]
if and only if $\scalar{P_1f}{f}_\pi\ge\scalar{P_2f}{f}_\pi$ 
for all $f\in L_2(\pi)$.
In the following let us motivate why the consideration of this ordering is 
particularly meaningful
for Markov chains.

\subsubsection{Consequences for the speed of convergence}
\label{sec:impl}

There are many ways to measure the speed of convergence of 
the distribution of a Markov chain towards its stationary distribution.
Probably the most desirable quantity is the \emph{total variation distance} 
of $\nu P^n$ and $\pi$, i.e.
\[
\|\nu P^n - \pi\|_{\rm tv} \,=\, \sup_{A\subseteq K} |\nu P^n(A) - \pi(A)|,
\]
where $\nu P^n(A)=\int_K P^n(x,A) \dint\nu(x)$ is the distribution 
of the Markov chain with transition kernel $P$ and initial distribution $\nu$ 
after $n$ steps.
However, estimating the total variation distance is quite delicate 
and, in practice, it is usually much easier to derive bounds on it by 
more analytic quantities, 
like isoperimetric constants or certain norms of $P$, 
see e.g.~\cite{Lo99,LoVe06} and \cite{MaNo07}.
Many of these quantities are defined by 
\[
c_\mathcal{M}(P) = \inf_{f\in\mathcal{M}}\, \scalar{(I-P)f}{f}_\pi
\]
for certain sets of functions $\mathcal{M}\subset L_2(\pi)$, 
where $If:=f$. 
Hence, the proof of $P_1\le P_2$  is enough to obtain 
$c_\mathcal{M}(P_1)\le c_\mathcal{M}(P_2)$ for every choice of $\mathcal{M}$.
Prominent examples are the
\begin{itemize}
	\item spectral gap: $\mathcal{M}=\{f\colon \|f\|_\pi=1,\,\int_K f\dint \pi=0 \}$
	\item conductance: $\mathcal{M}=\{f\colon 
		f=\frac{\ind_A}{\sqrt{\pi(A)}},\;\pi(A)\in(0,1/2],\;A\subset K \}$
	\item log-Sobolev constant: $\mathcal{M}=\{f\colon
		\int_K f^2 \log(f^2/\|f\|_\pi^2)\dint\pi=1\}$.
\end{itemize}
There are some more quantities of this form, like 
the best constant in a Nash inequality, 
average and blocking conductance. For details see e.g.~\cite{Ch05} and 
\cite{MoTe06}.

In what follows, we will prove $P_1\le P_2$ for a couple 
of Markov operators and, hence, that $P_2$ is ``faster'' than $P_1$ 
in all the above senses.

\subsubsection{Consequences for sample averages}
The property $P_1\le P_2$
has also consequences for the worst case mean square error and the 
asymptotic stationary variance of Markov chain Monte Carlo methods. 

Let $(X^{(i)}_k)_{k\in\N}$ be a Markov chain with 
transition kernel $P_i$, $i=1,2$, and initial distribution $\pi$, 
and define the 
\emph{Markov chain Monte Carlo method} 
\[
S_n^{(i)}(f) = \frac1n\sum_{k=1}^n f(X^{(i)}_k), \qquad i=1,2,
\]
which gives an approximation to $Ef:=\int_K f \dint\pi$ for $f\in L_2(\pi)$. 
By virtue of spectral theoretic arguments\footnote{One argues with 
Corollary~3.27 and Lemma~2.12 from \cite{Ru12},
as well as the fact that 
the spectral gap of $P_1$ is smaller than that of $P_2$.
} 
one can show that
\[
 \sup_{\norm{f}_\pi \leq 1} \E\left| S_n^{(2)}(f) - Ef \right|^2
 \leq  \sup_{\norm{f}_\pi \leq 1}  \E\left| S_n^{(1)}(f) - Ef \right|^2.
\]

Another interesting consequence is stated in \cite[Theorem~6]{MiLe09}. 
There it is proven that $P_1\leq P_2$ 
\emph{if and only if} 
\[
\mathbb{V}(f,P_2) \leq \mathbb{V}(f,P_1),
\]
for each individual $f\in L_2(\pi)$,
where 
$\mathbb{V}(f,P_i):= \lim_{n\to \infty} n \cdot \mathbb{E}\vert S_n^{(i)}(f)-E f \vert^2 $
is the asymptotic stationary variance 
(which can also be considered as asymptotic mean square error).

\subsection{The algorithms}
We present 
the different algorithms we consider 
in detail and provide relevant literature.

\subsubsection{Hit-and-run algorithm}
The hit-and-run algorithm proposed by Smith 
\cite{Sm84} is well studied in different 
settings, see
\cite{BeRoSm93,DiFr97,KaSm98,Lo99,LoVe06,KiSmZa11,Ru13}.

Informally it samples at each step on a randomly chosen $1$-dimensional 
line with respect to the 
corresponding conditional distribution. 
Let $\sphere$ be
the Euclidean unit sphere in $\R^d$.
For $x\in K$ and $\theta\in \sphere$ we define
\[
  L(x,\theta) = \set{x+s\theta \in K  \mid s\in \R}
\]
as the \emph{chord through $x$ in direction $\theta$}.
A transition from $x$ to $y$ of the hit-and-run algorithm works as follows:
Generate a set $L(x,\theta)$ by choosing $\theta$ with the uniform distribution on the sphere
and, then, choose $y \in L(x,\theta)$ according to the distribution determined by $\rho$ conditioned
on the chord $L(x,\theta)$. 
A transition of hit-and-run is given by Algorithm~\ref{alg: har}.
\begin{algorithm}\label{alg: har} (Hit-and-run)\\ 
Transition from current state $x$ to next state $y$:
\begin{enumerate}
 \item Sample $\theta \sim \mbox{Uniform}( \sphere)$.
 \item Sample $y\sim H_\theta(x,\cdot)$, where
 \[
  H_\theta(x,A) = \frac{\int_{L(x,\theta)} \mathbf{1}_A(z) \rho(z) \dint z}{\int_{L(x,\theta)} \rho(z)\dint z}.
 \]
\end{enumerate} 
\end{algorithm}
Note that the integral in $H_\theta$ is over a $1$-dimensional subset of $\R^d$ and the
integration is with respect to the $1$-dimensional Lebesgue measure.
For $x\in K$ and $A\subseteq K$ the transition kernel, say $H$, 
of the hit-and-run algorithm is determined by
\[
 H(x,A)= \int_{\sphere} H_\theta(x,A) \dint \sigma(\theta),
\]
where $\sigma=\mbox{Uniform}( \sphere)$ 
denotes the uniform distribution on the sphere.
It is well known that this transition kernel is reversible with respect to $\pi$, see for example \cite{BeRoSm93}.

An important special case of the hit-and-run algorithm above
is given if the density is an indicator function, 
for example $\rho=\mathbf{1}_K$. 
In this case, the hit-and-run algorithm is reversible with respect to
the uniform distribution on 
$K$.
Thus, under weak
regularity assumptions, the uniform distribution is the (unique) stationary distribution, 
see \cite{BeRoSm93}. 
We call this special case uniform hit-and-run.
Let us mention that we use the uniform hit-and-run in the next section.

\subsubsection{Simple and hybrid slice sampling}
Slice sampling belongs to the class of
auxiliary variable algorithms that are defined by a Markov chain
on an extended state space, see \cite{Hi98,RoRo99,MiMoRo01,MiTi02,RoRo02,Ne03}
and the references therein. 

We consider the simple slice sampler and 
a hybrid slice sampler based on hit-and-run.
A single transition of the simple slice sampler is presented in Algorithm~\ref{alg: SS}.
\begin{algorithm} \label{alg: SS}
(Simple slice sampler)\\
Transition from current state $x$ to next state $y$.
\begin{enumerate}
 \item Sample $t\sim \mbox{Uniform}(0,\rho(x))$.
 \item Sample $y\sim \mbox{Uniform}(K(t))$, where
 \[
  K(t) = \{ x\in K\mid \rho(x)> t  \}
 \]
 is the super-level set of $\rho$ determined by $t$.
\end{enumerate}
\end{algorithm}

The transition kernel, say $S$, corresponding to Algorithm~\ref{alg: SS} is
\[
 S(x,A) = \frac{1}{\rho(x)} \int_0^{\rho(x)} \frac{\vol_d(A\cap K(t))}{ \vol_d(K(t))} \dint t,
\]
where $\vol_d$ denotes the $d$-dimensional Lebesgue measure.
The simple slice sampler exhibits quite robust convergence properties, see \cite{RoRo99,MiTi02}.
However, a crucial drawback is that the second step is difficult to implement.
Because of this we consider the 
following hybrid slice sampler based on hit-and-run.
The idea is to replace the second step of the simple slice sampler by a 
Markov chain transition according to the uniform hit-and-run algorithm in $K(t)$,
see Algorithm~\ref{alg: HSS}.
\begin{algorithm}\label{alg: HSS}
(Hybrid slice sampler based on hit-and-run)\\
Transition from current state $x$ to next state $y$.
\begin{enumerate}
 \item Sample $t\sim \mbox{Uniform}(0,\rho(x))$ and $\theta\sim \mbox{Uniform}(\sphere)$ 
 independently.
 \item Sample $y\sim \mbox{Uniform}\,(L_t(x,\t))$, where
      \[
       L_t(x,\theta) = \{ x+r\t \in K(t) \mid r \in \R  \}
      \]
      is the chord through $x$ in direction $\t$ restricted to $K(t)$.
\end{enumerate}
\end{algorithm}
The transition kernel, say $U$, of this hybrid slice sampler is given by
\[
 U(x,A) = \frac{1}{\rho(x)} \int_0^{\rho(x)} \int_{\sphere} 
 \frac{\abs{L_t(x,\theta)\cap A}}{\abs{L_t(x,\theta)}} \dint \sigma(\theta	) \dint t
\]
where $\abs{\cdot}$ here denotes the $1$-dimensional Lebesgue measure.
This modification is tempting since the uniform hit-and-run algorithm is, at least in some scenarios, implementable.
For example, when the super-level sets are convex ($\rho$ is quasi-conave) by a bisection method
one can roughly approximate the intersection points of the line $x+r \theta$ with $K(t)$ and use a $1$-dimensional acceptance/rejection
approach to sample the uniform distribution on $L_t(x,\theta)$.
Further, the simple slice sampler and the hybrid slice sampler are reversible with respect to $\pi$,
see for example in \cite{LaRu14}.

\subsubsection{Random walk Metropolis}
The random walk Metropolis in $\R^d$ provides an easy to implement method 
for Markov chain sampling. Further, it is well studied  
and different convergence results are known, see for example \cite{RoTw96,MeTw96,JaHa00}.

To guarantee that a certain operator is positive we 
consider a \emph{lazy version of a random walk Metropolis}.
Further, we assume in the following that
 $q\colon \R^d \to [0,\infty) $ is a rotational invariant probability 
 density on $\R^d$, i.e. 
 $q(r\theta_1) = q(r \theta_2)$ for $\theta_1 ,\theta_2 \in \sphere$. 
 The rotational invariance guarantees that 
$q$ is symmetric. 
 Let us provide two examples which satisfy the rotational invariance.
 \begin{ex} \label{ex:bw}
   Let $x\in \R^d$, $\delta>0$ and $\ball$ be the Euclidean unit ball 
   with 
   $\kappa_d:=\vol_d(\ball)$.
   Then we can set
	$q(x):= \mathbf{1}_{\ball}(\delta^{-1} x)/(\delta^d \kappa_d)$
   and the corresponding random walk Metropolis is known as $\delta$-ball walk.
 \end{ex}
 \begin{ex}\label{ex:normal}
   Again, let $x\in \R^d$ and set $q(x):=\exp(-\abs{x}^2/2)/(2 \pi)^{d/2}$. 
   The corresponding random walk Metropolis is known as Gaussian random walk.
 \end{ex}
In Example~\ref{ex:bw} the proposal $q$ depends on a parameter $\delta$. In connection to this
we want to mention that in recent years 
optimal scaling results concerning a parameter of 
the proposal of the random walk Metropolis attracted 
a lot of attention, see \cite{RoGeGi97,RoRo01,NeRoYu12}. 
 
Now a single transition of the (lazy) random walk Metropolis with proposal $q$
is described in Algorithm~\ref{alg: RWM}.
\begin{algorithm} \label{alg: RWM}
(Random walk Metropolis)\\
Transition from current state $x$ to next state $y$ with proposal $q$.
\begin{enumerate}
 \item Sample $z\sim q$ and $u_1,u_2\sim \mbox{Uniform}\,(0,1)$ independently.
 \item If $x+z\in K$,  $u_1 \leq 1/2$ and $u_2<\min\{1,\rho(x+z)/\rho(x)\}$ 
 then accept the proposal, and set $y := x+z$, else
reject the proposal and set $y:=x$.
\end{enumerate}
 \end{algorithm}
To simplify the notation we define the acceptance probability
of a proposed state $y=x+z$ as
\[
\a(x,y) \;=\; \min\Bigl\{1,\, \frac{\rho(y)}{\rho(x)}\Bigr\}
\]
for $x, y\in K$ and $\a(x,y)=0$ otherwise. 
Then, the transition kernel 
is given by 
\[\begin{split}
M(x,A) \;&=\; \frac{1}{2}\, \int_{K} \ind_A(y)\; 
							\a(x,y)\, q(y-x) \dint y,
\end{split}\]
for $A\subseteq K$ with $x\notin A$ 
and $M(x,\{x\})= 1-M(x,K\setminus \{x\})$.

\section{Auxiliary variable Markov chains} \label{sec:aux}

We develop a systematic approach how to compare Markov chains
which can be described by a suitable two-step procedure. 
That many Markov chains are of this form was already 
observed in \cite{AnDi07} and 
the idea of a comparison of this type was developed in 
\cite{Ul12, Ul14} in a specific setting.

For this
let $\A$ be an arbitrary (index) set
equipped with 
a $\sigma$-finite measure $\lambda$.
By assumption
we have a function $s\colon K\times\A \to [0,\infty)$ which satisfies:
\begin{itemize}
 \item for all $x\in K$ we have that $s(x,\cdot)$ is a probability density function
 according to $\lambda$;
 \item for all $a\in\A$ we have that $s(\cdot,a)$ is integrable according to $\pi$.
\end{itemize}
Define for almost all $a\in\A$ 
(concerning $\lambda$) 
a probability measure $\pi_a$ on $K$
induced by $s(x,a)$, i.e. 
\[
 \pi_a(A) = \frac{\int_A s(x,a) \dint \pi(x)}{ \int_K s(x,a)\dint \pi(x) }, \qquad 
 A\subseteq K.
\]
In addition, we assume that
\begin{itemize}
 \item for every $a\in\A$ we have an equivalence relation 
$\sim_a$ and by $[x]_a=\{y\in K\colon x\sim_a y\}$ we denote the equivalence class 
of $K$ with respect to $\sim_a$ to which $x$ belongs;
\item for ($\lambda$-almost) every $a\in\A$ we have a transition kernel 
$P_a$ on $(K, \mathcal{B}(K))$, such that 
$P_a(x,A)=0$ for each $x\in K$ and $A\subseteq K\setminus[x]_a$.
\end{itemize}
With this we can define a Markov chain
in $K$ for which a single transition, 
starting from $x\in K$, can be written 
as the following procedure:
\begin{itemize}
	\item[1)] Sample $a\in\A$ according to the distribution 
	with density $s(x,\cdot)$.
	\item[2)] Generate the next state with respect to 
						$P_{a}(x,\cdot)$.
\end{itemize}
That is, the transition kernel $P$ is given by
\begin{equation}\label{eq:repr}
P(x,A) \,=\, \int_\A P_a\bigl(x, A\bigr)  s(x,a)\dint\lambda(a)
\end{equation}
for $A\subseteq K$.

\begin{remark}
Since the support of the measure $P_a(x,\cdot)$ is contained in 
$[x]_a$ and $[x]_a=[y]_a$ if $y\in[x]_a$, 
we can interpret a Markov chain with transition kernel $P_a$ 
on $K$ and initial state $x$ as a Markov chain on $[x]_a$.
\end{remark}

Clearly, for every Markov chain a transition can be written in this form. 
(Simply, set $P_a(x,\cdot):=P(x,\cdot)$ and take $[x]_a\equiv K$ for arbitrary $\A$ and $\lambda$.)
However,
there might also be more interesting equivalence relations and scenarios
as we illustrate for the hit-and-run transition kernel.
\begin{ex}  \label{ex: har}
Here let $\A:=\sphere$ and $\lambda:=\sigma$. Then, for every $a\in\A$ let $x\sim_a y$ if and only if 
$x\in L(y,a)$, such that $[x]_a := L(x,a)$.  
Further, for $(x,a)\in K\times \A$ set $s(x,a):=1$, 
which implies $\pi_a = \pi$ for all $a\in\A$.
For ($\lambda$-almost every) $a\in \A$ we have that
\[
 H_{a}(x,A) = \frac{\int_{[x]_a} \rho(y) \mathbf{1}_A(y) \dint y}{\int_{[x]_a} \rho(y)\dint y}, \qquad A\subseteq K,
\]
is a transition kernel on $K$ and
we can write the Markov operator $H \colon L_2(\pi) \to L_2(\pi)$ of the
hit-and-run algorithm in the form \eqref{eq:repr} as
\[
 H f(x) = \int_\A \int_K f(y) \,H_a(x,\!\dint y) \dint\lambda (a). 
\]
\end{ex}

In some cases it is even possible to 
represent the transitions of two different Markov chains in the form $\eqref{eq:repr}$ 
with the same equivalence relations $\sim_a$ and $s(x,a)$, 
so that the corresponding ``inner'' transition kernels are much easier 
to analyze.
We show that a suitable relation of these ``inner'' 
kernels is enough to compare the original Markov chains.

\begin{lemma}  \label{lemma:main}
Assume that for ($\lambda$-almost) all $a\in \A$ 
there are transition kernels $P^{(1)}_{a},P^{(2)}_{a}$ on $K$ such that
 \begin{enumerate}
  \item\label{it:main_one} $P^{(1)}_a$ and $P^{(2)}_a$ are self-adjoint operators on $L_2(\pi_a)$;
  \item $P^{(1)}_a$ is positive on $L_2(\pi_a)$;
  \item\label{it:main_three} $P^{(1)}_a P^{(2)}_a = P^{(2)}_a$.
 \end{enumerate}
Then, for the operators $P_{1},P_{2}\colon L_2(\pi) \to L_2(\pi)$ defined by
\begin{align*}
 P_{i}f(x) & = \int_\A P^{(i)}_a f(x)\, s(x,a)\dint \lambda(a), \quad i\in\{1,2\}
\end{align*}
holds $
 P_{1} \leq P_{2}
$.
\end{lemma}

\begin{proof}
 
 First, we show that $P_i$ can be written 
 as a product of suitable operators.
 For this let $\mu$ be a probability measure on $K \times \A$ that is given by 
 \begin{equation}\label{eq:mu}
\mu(B) \,:=\, \int_K \int_\A \,\ind_B(x,a)\, s(x,a)\dint\lambda(a) \dint\pi(x)
\end{equation}
for $B\subseteq K\times\A$, 
and let $L_2(\mu)$ be the Hilbert space of 
functions $g\colon K\times \A\to\R$ 
with finite norm $\norm{g}_\mu= \scalar{g}{g}_\mu ^{1/2}$,
where the inner-product of $g,h\in L_2(\mu)$ is defined by
\[\begin{split}
  \scalar{g}{h}_\mu 
	&= \int_{K\times\A} g(x,a)\, h(x,a) \dint\mu(x,a) \\
	&= \int_K \int_{\A} 
  g(x,a)\, h(x,a)\, s(x,a) \dint\lambda(a)\dint\pi(x).
\end{split}\]
Further, let $R \colon L_2(\mu) \to L_2(\pi)$ be the operator that is given by
\begin{equation}\label{eq:R}
R g(x) \,=\,  \int_{\A} g(x,a)\,s(x,a) \dint\lambda(a), \qquad g\in L_2(\mu).
\end{equation}
Since
\[
\scalar{f}{R g}_\pi \,=\,  \int_K \int_{\A} 
	f(x)\, g(x,a)\,s(x,a) \dint\lambda(a) \dint\pi(x),
\]
we obtain that the adjoint operator $R^*\colon L_2(\pi) \to L_2(\mu)$ satisfies 
\begin{equation} \label{eq:R*}
R^*f(x,a) \,=\, f(x), \qquad f\in L_2(\pi), 
\end{equation}
for $a\in\A$ and $x\in K$. 
It is useful to write $g_a(x)=g(x,a)$ for $g\in L_2(\mu)$ and fixed $a\in \A$. 
Clearly, $g_a\in L_2(\pi_a)$ for almost every $a\in\A$ (with respect to~$\lambda$), 
such that $P^{(i)}_a g_a$ is well-defined. Let
\begin{equation}\label{eq:P_tilde}
\wt P_i\, g(x,a) \,=\, \int_K g(y,a) \,P^{(i)}_a(x,\!\dint y), \qquad g\in L_2(\mu),
\end{equation}
and note that the operator satisfies $\wt P_i\colon L_2(\mu)\to L_2(\mu)$. 
An immediate consequence of the definitions is 
\begin{equation}\label{eq:prod}
P_i \,=\, R \wt P_i R^*, \qquad i\in\{1,2\}.
\end{equation}

We 
prove different properties of $\wt P_1$ and $\wt P_2$.
For $i\in \{1,2\}$ we show that $\wt P_i$ is self-adjoint on $L_2(\mu)$.
Note that by the assumptions we know that $P^{(i)}_a$ is self-adjoint on $L_2(\pi_a)$.
With 
\[
 C_a= \int_K s(x,a)\dint \pi(x)
\]
we have for $g,h\in L_2(\mu)$ that
\begin{align*}
 \scalar{\wt P_i g}{h}_\mu 
 & = \int_K \int_{\A} \int_K g(y,a)\,P^{(i)}_a(x,\!\dint y)\, h(x,a)\, s(x,a)\dint \lambda(a)\dint \pi(x)\\
 & = \int_\A C_a \int_K \int_K g(y,a)\,P^{(i)}_a(x,\!\dint y)\, h(x,a)\,\dint \pi_a(x) \dint\lambda(a)\\
 & = \int_\A C_a \scalar{P^{(i)}_a g_a}{h_a}_{\pi_a} \dint \lambda(a)
   = \int_\A C_a \scalar{ g_a}{P^{(i)}_a h_a}_{\pi_a} \dint \lambda(a)\\
 & = \scalar{g}{\wt P_i h}_\mu.
 \end{align*}
By the same line of arguments, for $g\in L_2(\mu)$  we have
\begin{align*}
  \scalar{\wt P_1 g}{g}_\mu = \int_\A C_a \scalar{P^{(1)}_a g_a}{g_a}_{\pi_a} \dint \lambda(a)
  \geq 0,
\end{align*}
such that $\wt P_1$ preserves the positivity of $P_a^{(1)}$, i.e. $\wt P_1$ is positive on $L_2(\mu)$. 

Further, for $g\in L_2(\mu)$ we obtain
\begin{align*}
   \wt P_1 \wt P_2\, g (x,a) = P^{(1)}_a P^{(2)}_a g_a(x) =  P^{(2)}_a  g_a(x) = \wt P_2\, g(x,a).
\end{align*}
By the self-adjointness of $\wt P_1$ and $\wt P_2$ we also have $\wt P_2 \wt P_1 = \wt P_2$.

Now we gathered all tools together to prove the assertion.
By the positivity, we know that $\wt P_1$ has a unique positive square root $N$, 
i.e. $N^2=\wt P_1$. It is well known that $N$ commutes with every 
operator that commutes with 
$\wt P_1$
, see e.g.~\cite[Theorem~9.4-2]{Kr89}, 
in particular $N\wt P_2=\wt P_2 N$. 
We obtain
\[\begin{split}
\scalar{P_2f}{f}_\pi &= \scalar{R\wt P_2 R^*f}{f}_\pi = \scalar{\wt P_2 R^*f}{R^*f}_\mu \\
&= \scalar{\wt P_1\wt P_2 R^*f}{R^*f}_\mu = \scalar{\wt P_2 N R^*f}{N R^*f}_\mu \\
&\le \scalar{N R^*f}{N R^*f}_\mu = \scalar{\wt P_1 R^*f}{R^*f}_\mu 
= \scalar{P_1 f}{f}_\pi, 
\end{split}\]
where the inequality comes from $\|\wt P_2\|\le1$, which is true since $\wt P_2$ is a 
Markov operator.\\
\end{proof}
Thus,
for each of the intended comparisons, say between $P_1$ and $P_2$, we have to
\begin{itemize}
  \item find representations 
	\[
	 P_i f(x) = \int_\A P_a^{(i)} f(x) s(x,a)\dint \lambda(a), \quad i\in \{1,2\};
	\]
  \item check the assumptions \ref{it:main_one}.-\ref{it:main_three}. of Lemma~\ref{lemma:main}, 
  	i.e.~reversibility of $P_a^{(i)}$ 
  	with respect to $\pi_a$, positivity of $P^{(1)}_a$ 
  	and $P^{(1)}_a  P^{(2)}_a = P^{(2)}_a$.      
\end{itemize}

\bigskip

\begin{remark}
Let us comment on the necessity and possible generalizations of 
the assumptions in Lemma~\ref{lemma:main}.
The second assumption, i.e., the positivity of $P^{(1)}_a$, is most likely 
an artifact of the proof technique and we conjecture that the result holds 
without this assumption. However, in the proof it is essential and we were not able to remove it. 
The same problem already appeared in~\cite{Ul12,Ul14}. 
The third assumption of the lemma says, roughly speaking, 
that a step of the Markov chain with corresponding operator $P^{(1)}_a$ 
``cannot be seen'' if followed by $P^{(2)}_a$. 
As the last steps of the proof show, this could be replaced, 
e.g., by the weaker assumption that $\wt P_1 \le \wt P_2$, i.e., 
$\scalar{\wt P_2 g}{g}_\mu \le \scalar{\wt P_1 g}{g}_\mu$ 
for all $g\in L_2(\mu)$. 
However, at least in the examples we consider, 
the relatively easy-to-check assumption 
\ref{it:main_three}.
is already fulfilled.
\end{remark}

\section{Main result}\label{sec:main}
In this section we apply Lemma~\ref{lemma:main} to prove the following theorem.
\begin{theorem}  \label{thm:main}
Let $M, U, H, S$ be the Markov operators of the (lazy) random walk Metropolis with rotational invariant
proposal $q$, the hybrid slice sampler based on hit-and-run, the hit-and-run algorithm and the simple slice sampler.
Then
\begin{align*}
  M & \leq U \leq H,\\
  M & \leq U \leq S.
 \end{align*}
\end{theorem}
Before proving the different inequalities in the statement 
we start with recalling some notion
and state a useful lemma.
Recall that $\sphere$ denotes the Euclidean unit sphere in $\R^{d}$
and $\sigma$ is the uniform distribution on $\sphere$.
For $\theta \in \sphere$, $x\in \R^d$ and $t\in [0,\infty)$ let
\[
 K(t)=\{ x\in K\mid \rho(x)> t\}
\]
be the super-level set of $\rho$ determined by $t$ 
and 
\[
 L_t(x,\theta) = \{ x+\theta r\in K(t) \mid r\in \R\}
\]
be the chord in $K(t)$ through $x$ in direction $\theta$.
Moreover, define  $L(x,\theta):=L_0(x,\theta)$ and, 
for a set $A\subseteq K = K(0)$, let
 \[
  \Pi_{\theta}(A) = \{y \in \R^d \mid y \bot \theta,\; 
  L(y,\theta)\cap A \neq \emptyset \}
 \]
be the orthogonal projection of $A$ to the hyperplane that is orthogonal to $\theta$.
We obtain the following useful results by an
application of the Theorem of Fubini and the integral transformation to polar coordinates.

\begin{lemma}
 Let $t\ge0$ and $\theta \in \sphere$. 
For Lebesgue integrable $f\colon K(t) \to \R$ we have 
 \begin{equation}  \label{eq:fubini}
   \int_{K(t)} f(x)\dint x = \int_{\Pi_{\theta}(K(t)) } 
   \int_{L_t(x,\theta)} f(y)\dint y \dint x,
 \end{equation}
and 
for fixed $x\in \R^d$ holds
 \begin{equation} \label{eq:polar}
  \int_{K(t)} f(y)\dint y = 
  \frac{d\kappa_d}{2 } \int_{\sphere} \int_{L_t(x,\theta)} f(y) \abs{x-y}^{d-1} 
  \dint y \dint \sigma(\theta)
 \end{equation}
 where $\kappa_d=\vol_d(\ball)$ denotes the volume of the $d$-dimensional Euclidean unit ball.
\end{lemma}
 Note that in both identities the inner integral on the ride-hand-side is over
 a $1$-dimensional subset of $\R^d$ and the integration is with respect to
 the $1$-dimensional Lebesgue measure.

The different inequalities in Theorem~\ref{thm:main} will be proven in the following sections.
There we use the notation of Section~\ref{sec:aux}.

\subsection{Hit-and-run vs.~hybrid slice sampler:}\label{subsec:har}
For the hit-and-run we use the scenario described in Example~\ref{ex: har}.
Hence, let $\A:=\sphere$ and $\lambda:=\sigma$. 
Further, for $(x,a)\in K\times \A$ set $s(x,a):=1$ such as 
$
 [x]_a := L_0(x,a).
$
This implies $\pi_a = \pi$ for all $a\in\A$.
By 
\[
 H_{a}(x,A) = \frac{\int_{[x]_a} \rho(y) \mathbf{1}_A(y) \dint y}{\int_{[x]_a} \rho(y)\dint y}, \qquad A\subseteq K,
\]
we can write the Markov operator $H \colon L_2(\pi) \to L_2(\pi)$ of the
hit-and-run algorithm as
\[
 H f(x) = \int_\A \int_K f(y) \,H_a(x,\!\dint y) \dint\lambda (a). 
\]
Let
\[
 U_a(x,A) = \frac{1}{\rho(x)} \int_0^{\rho(x)} \int_{L_t(x,a)} \, \frac{\mathbf{1}_A(y)}{\abs{L_t(x,a)}}\, \dint y \dint t
\]
and observe that $L_t(x,a) = [x]_a \cap K(t)$.
With this we can write  
the Markov operator $U \colon L_2(\pi) \to L_2(\pi)$ of the hybrid slice sampler based on hit-and-run as
\[
 U f(x) = \int_A \int_K f(y)\, U_a(x,\!\dint y) \dint \lambda(a).  
\]
Thus we have a common representation.
Now we check the assumptions \eqref{it:main_one}.-\eqref{it:main_three}. 
of Lemma~\ref{lemma:main}:
\begin{enumerate}
 \item To the reversibility of $H_a$ and $U_a$ with respect to $\pi_a$:\\
 We only show reversibility of $U_a$, since reversibility of $H_a$ follows by the same line of arguments.
 For $A,B\subseteq K$
 it is enough to prove
 \begin{align} \label{al:to_prove}
      \int_A U_a(x,B)\rho(x) \dint x = \int_B U_a(x,A) \rho(x) \dint x.
 \end{align}
 By the application of \eqref{eq:fubini} with $t=0$, 
 the equality $\mathbf{1}_{[0,\rho(y))}(t)=\mathbf{1}_{K(t)}(y)$ 
 and the fact that $y\in [x]_a$ implies $L_t(y,a)=L_t(x,a)$ we obtain
 \[\begin{split}
 \int_A U_a&(x,B) \rho(x) \dint x
   = \int_{\Pi_{a}(K) }\int_{[x]_a} \mathbf{1}_A(y) U_a(y,B) \rho(y)  \dint y \dint x \\
 & = \int_{\Pi_{a}(K)}\int_{[x]_a} \mathbf{1}_A(y) \int_0^{\rho(y)} \int_{L_t(y,a)} \, \frac{\mathbf{1}_B(z)}{\abs{L_t(y,a)}}\,  \dint z \dint t \dint y \dint x  \\
& = \int_{\Pi_{a}(K)} \int_0^{\infty} \frac{1}{\abs{L_t(x,a)}} \,
	\abs{L_t(x,a)\cap A}\,\abs{L_t(x,a)\cap B}\, \dint t\dint x.
\end{split}\]
 Observe that in the right-hand-side $A$ and $B$ are interchangeable 
which proves \eqref{al:to_prove}.
\item To the positivity of $U_a$: \\
 By similar arguments as in the proof of \eqref{al:to_prove} we obtain
 with $c=\int_K \rho(x)\dint x$ that
 \begin{align*}
  \scalar{U_a f}{f}_\pi 
 & =\frac{1}{c} \int_{\Pi_{a}(K)} \int_0^{\infty} \frac{1}{\abs{L_t(x,a)}} \left(\int_{L_t(x,a)} f(y)\dint y \right)^2 \dint t \dint x
 \geq 0. 
 \end{align*}
\item To $U_a H_a = H_a$:\\
	From the definition of the Markov kernels $H_a$ it is obvious that 
	$y\in [x]_a$ implies $H_a(y,A)=H_a(x,A)$ for all $A\subseteq K$. 
	This implies 
	\[
	U_a H_a(x,A) = \frac{1}{\rho(x)} \int_0^{\rho(x)} \int_{L_t(x,a)} H_a(y,A) \frac{\dint y}{\abs{L_t(x,a)}} \dint t 
	= H_a(x,A)
	\]
	and hence $U_a H_a = H_a$.
\end{enumerate}
Thus, as a direct consequence of Lemma~\ref{lemma:main} we obtain $U\leq H$.

\subsection{Simple vs.~hybrid slice sampler:}

Here, we derive another representation of the hybrid slice sampler
adapted to simple slice sampling.

For this let $\A := [0,\infty)$ 
and $\lambda$ be the $1$-dimensional Lebesgue measure.
Further, for $(x,a)\in K\times \A$ set 
$s(x,a):=\mathbf{1}_{[0,\rho(x))}(a)/\rho(x)$ such as $[x]_a := K(a).$
Observe that
\[
 \pi_a(A)  = \frac{\vol_d(A\cap K(a))}{\vol_d(K(a))}.
\]
For any $x\in K$ let
\[
 S_{a}(x,A) = \pi_a(A),\quad A\subseteq K.
\]
Clearly, the Markov operator $S \colon L_2(\pi) \to L_2(\pi)$ of the
simple slice sampler can be written as
\[
 S f(x) = \int_\A \int_K f(y) \,S_a(x,\!\dint y)\, s(x,a) \dint\lambda (a). 
\]
Note that, with these notations, we have $L_a(x,\theta)=[x]_a \cap L(x,\theta)$ 
and let
\[
 U_a(x,A) = \int_{\sphere} \int_{L_a(x,\theta)} 
  \frac{\mathbf{1}_A(y)}{\abs{L_a(x,\theta)}}\,\dint y \dint \sigma(\theta).
\]
With this we can write  
the Markov operator $U \colon L_2(\pi) \to L_2(\pi)$ of the hybrid slice sampler based 
on hit-and-run as
\[
 U f(x) = \int_\A \int_K f(y)\, U_a(x,\!\dint y) s(x,a) \dint \lambda(a).  
\]
Again, we have a common representation and it remains to check
the assumptions 
\ref{it:main_one}.-\ref{it:main_three}. of Lemma~\ref{lemma:main}:
\begin{enumerate}
 \item To the reversibility of $S_a$ and $U_a$ with respect to $\pi_a$:\\
 Since $S_a(x,A)=\pi_a(A)$ reversibility of $S_a$ is obvious.
 Observe that $\pi_a$ is the uniform distribution on $[x]_a=K(a)$ and
 $U_a(x,\cdot)$ performs a uniform hit-and-run step on $K(a)$ which is known to be reversible, 
 see for example \cite[Lemma~4.10]{Ru12}.
 \item To the positivity of $U_a$:\\
  With $c=\vol_d(K(a))$, by the application of \eqref{eq:fubini} 
  and by the fact that for $y\in L_a(x,\theta)$ follows $L_a(x,\theta) =L_a(y,\theta)$ holds
  \begin{align*}
   \scalar{U_a f&}{f}_{\pi_a} 
    = \frac{1}{c} \int_{K(a)}  \int_{\sphere} \int_{L_a(x,\theta)} 
   \frac{f(z)f(x) }{\abs{L_a(x,\theta)}}\dint z \dint \sigma (\theta) \dint x \\
   & = \frac{1}{c} \int_{\sphere} \int_{\Pi_{\theta}(K(a))} 
   \int_{L_a(x,\theta)} 
    \int_{L_a(y,\theta)}  \frac{f(y)f(z)}{\abs{L_a(y,\theta)}} \dint z \dint y \dint x \dint \sigma(\theta) \\
  & = \frac{1}{c} \int_{\sphere} \int_{\Pi_{\theta}(K(a))} 
  \frac{1}{\abs{L_a(x,\theta)}}
  \left(\int_{L_a(x,\theta)} f(y)
       \dint y \right)^2 \dint x \dint \sigma(\theta)
       \geq 0.
    \end{align*}

 \item To $U_a S_a = S_a$:\\
  This follows immediately from $S_a=\pi_a$ and the reversibility of $U_a$ with respect to $\pi_a$.
\end{enumerate}
Thus, as a direct consequence of Lemma~\ref{lemma:main} we obtain $U\leq S$.

\subsection{Hybrid slice sampler vs.~Metropolis:}

Again we need a suitable representation for the 
hybrid slice sampler based on hit-and-run.

Here, let $\A := \sphere \times [0,\infty)$ and $\lambda$
be the product measure of $\sigma$ and the $1$-dimensional Lebesgue measure.
Further, for $x\in K$ and $(a_1,a_2)\in \A$, set $s(x,a_1,a_2) = \mathbf{1}_{[0,\rho(x)]}(a_2)/\rho(x)$
such as $[x]_{(a_1,a_2)}:= L_{a_2}(x,a_1)$.
Observe that
\[
 \pi_{(a_1,a_2)} (A) = \frac{\vol_d(A\cap K(a_2))}{\vol_d(K(a_2))}, \quad A\subseteq K, 
\]
is the uniform distribution in $K(a_2)$. 
By
\[
 U_{(a_1,a_2)}(x,A) = \int_{[x]_{(a_1,a_2)}} \mathbf{1}_A(y) \frac{\dint y}{\abs{[x]_{(a_1,a_2)}}}
\]
we have a representation 
of the Markov operator $U\colon L_2(\pi) \to L_2(\pi)$ of the hybrid slice sampler by
\[
 Uf(x) = \int_\A \int_K f(y) U_{a}(x,\!\dint y) s(x,a) \dint \lambda (a)
\]
with $a\in\A$.
Now we have to represent the random walk Metropolis in the same fashion.
For this let us define
\[
 \eta(x,y) = \frac{d \kappa_d}{2}\, q(y-x)\,\abs{y-x}^{d-1}, \quad x,y\in \R^d
\]
and for $A\subseteq K$ let
\[
 M_{a}(x,A) 
 = \frac{1}{2} \int_{[x]_{a}} \mathbf{1}_{A}(y)\, \eta(x,y) \dint y
 +\mathbf{1}_A(x)\left( 1-\frac{1}{2} \int_{[x]_{a}} \eta(x,y) \dint y \right).
\]
Here it is essential that $q$ is rotational invariant, namely this property assures
that $\int_{[x]_{a}} \eta(x,y) \dint y \leq 1$.
Note that $M_a$ is again a lazy transition kernel, since $M_a(x,\{x\}) \geq 1/2$.
For $x\not \in A$ by \eqref{eq:polar} we have
\begin{align*}
 M(x,A) 
 & = \frac{1}{2}\,\int_{K} \mathbf{1}_A(y)\,\a(x,y)\, q(y-x) \dint y\\
 & = \frac{1}{2 \rho(x)}\, 
     \int_0^{\rho(x)} \int_{K} \mathbf{1}_{[0,\rho(y)]}(t) \mathbf{1}_A(y)\,q(y-x)  \dint y \dint t \\
 & = \frac{1}{2 \rho(x)}\, 
     \int_0^{\rho(x)} \int_{K(t)}  \mathbf{1}_A(y)\,q(y-x)  \dint y \dint t\\
 & = \frac{d\kappa_d}{4 \rho(x)} \int_{\sphere} \int_0^{\rho(x)} \int_{L_t(x,\theta)} 
    \mathbf{1}_A(y) q(x-y)\abs{x-y}^{d-1} \dint y \dint t\dint \sigma(\theta) \\
 & = \int_\A M_a(x,A) s(x,a) \dint \lambda(a).   
\end{align*}
Thus, the transition kernel of the random walk Metropolis has the desired representation.
It remains to check
the conditions \ref{it:main_one}.-\ref{it:main_three}. of Lemma~\ref{lemma:main}:
\begin{enumerate}
 \item To the reversibility of $M_a$ and $U_a$ with respect to $\pi_a$:\\
 This follows again by a suitable application of \eqref{eq:fubini}. 
 Let $a=(a_1,a_2)$. Due to its simple form, reversibility of $U_a$ is obvious. 
 It is enough to show for disjoint $A,B\subseteq K$ that 
 \[
  \int_{A\cap K(a_2)}  M_a(x,B)\dint x = \int_{B\cap K(a_2)} M_a(x,A) \dint x.
 \] 
 Since $L_{a_2}(x,a_1) = L_{a_2}(y,a_1)$ if $y\in L_{a_2}(x,a_1) $ we have
 \begin{align*}
\frac{1}{2} &\int_{A\cap K(a_2)} \int_{[x]_a} \mathbf{1}_B(y) \eta(x,y)\dint y \dint x \\
 &  = \frac{1}{2} \int_{K(a_2)} \int_{L_{a_2}(x,a_1)} \mathbf{1}_A(x) \mathbf{1}_B(y) \eta(x,y) \dint y \dint x \\
 &  = \frac{1}{2} \int_{\Pi_{a_1}(K(a_2))} \int_{L_{a_2}(x,a_1)} \int_{L_{a_2}(y,a_1)} \mathbf{1}_A(y) \mathbf{1}_B(z) \eta(y,z) \dint z \dint y \dint x \\
 &  = \frac{1}{2} \int_{\Pi_{a_1}(K(a_2))} \int_{L_{a_2}(x,a_1)} \int_{L_{a_2}(x,a_1)} \mathbf{1}_A(y) \mathbf{1}_B(z) \eta(y,z) \dint z \dint y \dint x.
 \end{align*}
 By the symmetry of $q$, i.e. $q(y-z)=q(z-y)$ follows $\eta(y,z)=\eta(z,y)$. This leads to the reversibility.
\item To the positivity of $M_a$:\\
 By the fact that $M_a$ is a lazy transition kernel we have positivity.
 \item To $M_a U_a = U_a$:\\
 By the fact that $y\in [x]_a$ implies $[x]_a = [y]_a$ and hence $U_a(x,A)=U_a(y,A)$ for 
 all $A\subseteq K$ we obtain 
 \begin{align*}
 M_a U_a&(x,A) \\
 & = \frac{1}{2} \int_{[x]_a} U_a(y,A)\, \eta(x,y) \dint y + U_a(x,A)\left(1-\frac{1}{2}\int_{[x]_a} \eta(x,y) \dint y\right)\\
 & = U_a(x,A).   
 \end{align*}
\end{enumerate}
Thus, as a direct consequence of Lemma~\ref{lemma:main} we obtain $ M\leq U$.

\section{Concluding remarks} \label{sec:con_rem}
 
In this article we have presented a technique to compare the efficiency of Markov chains of a specific type.
Using this technique we provide two comparison hierarchies according 
to a partial ordering of Markov operators of four prominent Markov chains 
for sampling general distributions in $\R^d$. 
The comparison with respect to the partial ordering leads to 
comparison results according to different criterions, 
for example the spectral gap, the conductance or the log-Sobolev constant, cf.~Section~\ref{sec:impl}.

Let us mention here that the computational cost for the simulation of each individual Markov chain
is not taken into account. There seems to be a trade-off between efficiency 
and computional cost which should be further investigated. We leave this open for future work.  

Finally, there are three open problems related to the considered Markov chains. 

First, what is about the relation of the hit-and-run and the simple slice sampler? 
It is easy to see that there cannot be a general result as in the other cases. For this 
consider two examples: 
\begin{enumerate}
 \item If $\pi$ is the uniform distribution on $K$, then one step 
of the simple slice sampler is enough to sample $\pi$, while for the hit-and-run algorithm 
it is not (as long as $d>1$). Hence, in this situation, hit-and-run is worse. 
 \item Let $d=1$. Then, the hit-and-run algorithm samples $\pi$ in one step, regardless of $\pi$. 
Hence, in this situation hit-and-run is better.
\end{enumerate}
It seems to be interesting to find cases where hit-and-run is better. 
This is because, 
we guess that in general the computational cost of the simple slice sampler is 
(if it is at all implementable) much higher.

The second open problem concerns reverse inequalities. That is, if a Markov chain is better 
than another, how much better is it? This seems to be a delicate question and we can answer 
it only for toy examples. The techniques that were used in~\cite{Ul14} in a discrete setting 
do not seem to work here.

The last problem we want to mention: 
Is there a similar hierarchy for the \emph{mixing time}? That is, the number of steps that 
are needed to make the total variation distance ``small'', cf.~Section~\ref{sec:impl}. 
Certainly, the answer to this question additionally depends on the choice of the initial 
distribution. But the (quite analytical) techniques of this paper, 
see also~\cite{Pe73,LaRu14,Ul14,AnVi14}, 
do not seem to be suitable for this purpose. 
One interesting approach in this direction 
for Markov chains on discrete state spaces is given in 
\cite{FiKa13}.

\end{document}